\documentclass{amsart}
\usepackage{amssymb}
\usepackage{amsmath}
\usepackage{amsfonts}

\newtheorem{theorem}{Theorem}
\theoremstyle{plain}

\newtheorem{proposition}{Proposition}
\newtheorem{remark}{Remark}

\numberwithin{equation}{section}
\input{tcilatex}

\begin{document}
\title{On the second-order tangent bundle with deformed 2-nd lift metric}
\author{Abdullah MA\u{G}DEN}
\address{Ataturk University, Faculty of Science, Department of Mathematics,
25240, Erzurum-Turkey.}
\email{amagden@atauni.edu.tr}
\author{Kubra KARACA}
\address{Ataturk University, Faculty of Science, Department of Mathematics,
25240, Erzurum-Turkey.}
\email{kubrakaraca91@gmail.com}
\author{Aydin GEZER}
\address{Ataturk University, Faculty of Science, Department of Mathematics,
25240, Erzurum-Turkey.}
\email{agezer@atauni.edu.tr}

\begin{abstract}
Let $(M,g)$ be a pseudo-Riemannian manifold and $T^{2}M$ be its the
second-order tangent bundle equipped with the deformed $2-$nd lift metric $%
\overline{g}$ which obtained from the $2-$nd lift metric by deforming the
horizontal part with a symmetric $(0,2)-$tensor field $c$. In the present
paper, we first compute the Levi-Civita connection \ and its Riemannian
curvature tensor field of $(T^{2}M,g)$. We give necessary and sufficient
conditions for $(T^{2}M,g)$ to be semi-symmetric. Secondly, we show that $%
(T^{2}M,g)$ is a plural-holomorphic $B-$manifold with the natural integrable
nilpotent structure. Finally, we get the conditions under which $(T^{2}M,g)$
with the $2-$nd lift of an almost complex structure is an anti-K\"{a}hler
manifold.


\end{abstract}

\maketitle

\section{\protect\bigskip Introduction}

\noindent Given an $n-$dimensional manifold $M$, a second-order tangent
bundle $T^{2}M$ over $M$ can be constructed from the equivalent classes of
curves on $M$ which agree up to their acceleration (for details, see \cite%
{Dodson} and \cite{YanoIshihara}). Moreover, in \cite{Dodson}, it is proved
that a second-order tangent bundle $T^{2}M$ becomes a vector bundle over $M$
if and only if $M$ has a linear connection. The prolongations of tensor
fields and connections given on $M$ to its second-order tangent bundle $%
T^{2}M$ was studied in \cite{YanoIshihara}. Let $(M,g)$ be an $n-$%
dimensional pseudo-Riemannian manifold and $T^{2}M$ be its second-order
tangent bundle. The $2-$nd lift metric on $T^{2}M$ was defined and studied
by Yano and Ishihara in \cite{YanoIshihara}.

We point out here and once that all geometric objects considered in this
paper are supposed to be of class $C^{\infty }$. In this section, we recall
some fundamental facts on the second-order tangent bundle that are needed
later.

The second-order tangent bundle $T^{2}M$ of a differentiable manifold $M$ is
the $3n-$dimensional manifold as the set of all $2-$jets of $M$ determined
by mappings of the real line $%
\mathbb{R}
$ into $M$. The canonical projection $\pi _{2}:$ $T^{2}M\longrightarrow M$
defines the natural bundle structure of $T^{2}M$ over $M$. If we introduce
the canonical projection $\pi _{12}:$ $T^{2}M\longrightarrow TM$, then $%
T^{2}M$ has a bundle structure over the tangent bundle $TM$ with projection $%
\pi _{12}$. In the paper, we use Einstein's convention on repeated indices.

Let $(U,x^{i})$ be a system of coordinates in $M$ and $F$\ be a curve in $U$
which locally expressed as $x^{i}=F^{i}(t)$. If we take a $2-$jet $j^{2}F$
belonging to $\pi _{2}^{-1}(U)$ and define 
\begin{equation*}
x^{i}=F^{i}(0),\text{ }y^{i}=\frac{dF^{i}}{dt}(0),\text{ }z^{i}=\frac{1}{2}%
\frac{d^{2}F^{i}}{dt^{2}}(0),
\end{equation*}%
then the $2-$jet $j^{2}F$ is expressed uniquely by the set $%
(x^{i},y^{i},z^{i})$. Thus, $(x^{i},y^{i},z^{i})$ is the system of
coordinates induced in $\pi _{2}^{-1}(U)$ from $(U,x^{i})$. The coordinates $%
(x^{i},y^{i},z^{i})$ in $\pi _{2}^{-1}(U)$ are called the induced
coordinates and sometimes denote them by $\{\xi ^{A}\}$, that is, by putting%
\begin{equation*}
\xi ^{i}=x^{i},\text{ }\xi ^{\overline{i}}=y^{i},\text{ }\xi ^{\overline{%
\overline{i}}}=z^{i}.
\end{equation*}%
The indices $A,B,C,...$ run over the range $%
\{1,2,...,n;n+1,n+2,...,2n;2n+1,2n+2,...,3n\}$.

For a function $f$ locally expressed by $f=f(x)$ on $M$, there corresponds
on $T^{2}M$ the $0-$th, the $1-$st and the $2-$nd lifts of the function $f$
respectively defined by%
\begin{equation}
^{0}f=f(x);\text{ }^{I}f=y^{i}\partial _{i}f(x);\text{ }^{II}f=z^{i}\partial
_{i}f(x)+\frac{1}{2}y^{j}y^{i}\partial _{j}\partial _{i}f(x)  \label{AA2.0}
\end{equation}%
with respect to the induced coordinates $\{\xi ^{A}\}$, where $\partial _{i}=%
\frac{\partial }{\partial x^{i}}$.

Given a vector field $X=X^{i}\partial _{i}$ on $M$, we define the the $0-$%
th, the $1-$st and the $2-$nd lifts of $X$ to $T^{2}M$ as follows: \cite%
{YanoIshihara}%
\begin{equation}
{}^{0}X=X^{j}\partial _{\overline{\overline{j}}},  \label{AA2.1}
\end{equation}%
\begin{equation}
{}^{I}X=X^{j}\partial _{\overline{j}}+y^{s}\partial _{s}{X}^{j}\partial _{%
\overline{\overline{j}}}  \label{AA2.2}
\end{equation}%
and%
\begin{equation}
^{II}X=X^{j}\partial _{j}+y^{s}\partial _{s}{X}^{j}\partial _{\overline{j}%
}+\{z^{s}\partial _{s}X^{j}+\frac{1}{2}y^{t}y^{s}\partial _{t}\partial
_{s}X^{j}\}\partial _{\overline{\overline{j}}}  \label{AA2.3}
\end{equation}%
with respect to the induced coordinates $\{\xi ^{A}\}$, where $\partial _{i}=%
\frac{\partial }{\partial x^{i}},$ $\partial _{\overline{i}}=\frac{\partial 
}{\partial y^{i}},$ $\partial _{\overline{\overline{i}}}=\frac{\partial }{%
\partial z^{i}}$. By (\ref{AA2.1})-(\ref{AA2.3}) and (\ref{AA2.0}) we have
directly%
\begin{eqnarray*}
{}^{0}X^{0}f &=&0\text{,}^{0}X^{I}f=0\text{,}^{0}X^{II}f=^{0}(Xf), \\
^{I}X^{0}f &=&0\text{,}^{I}X^{I}f=^{0}(Xf)\text{,}^{I}X^{II}f=^{I}(Xf), \\
^{II}X^{0}f &=&^{0}(Xf)\text{,}^{II}X^{I}f=^{I}(Xf)\text{,}%
^{II}X^{II}f=^{II}(Xf)
\end{eqnarray*}%
for any vector field $X$ and function $f$ on $M$.

For the Lie bracket on $T^{2}M$ in terms of the lifts of vector fields $X,$ $%
Y$ on $M$, we have the following formulas: \cite{YanoIshihara}%
\begin{equation}
\left\{ 
\begin{array}{l}
{\left[ {}^{0}X,{}^{0}Y\right] =0}, \\ 
{\left[ {}^{0}X,{}{}^{I}Y\right] =0,}\text{ }{\left[ {}^{II}X,{}{}^{0}Y%
\right] ={}}^{0}{\left[ {}X,Y\right] ,}\text{ }{\left[ {}^{II}X,{}{}^{II}Y%
\right] =}^{II}{\left[ {}X,Y\right] } \\ 
{\left[ {}^{I}X,{}{}^{I}Y\right] =}^{0}{\left[ {}X,Y\right] ,}\text{ }{\left[
{}^{II}X,{}{}^{I}Y\right] ={}}^{I}{\left[ {}X,Y\right] .}%
\end{array}%
\right.  \label{AA2.4}
\end{equation}

\begin{remark}
The $2-$nd lift defined by (\ref{AA2.3}) determines an isomorphism of the
Lie algebra of vector fields on $M$ into the Lie algebra of vector fields on 
$T^{2}M$.
\end{remark}

\section{The deformed $2-$nd lift metric on the second-order tangent bundle}

The $0-$th, the $1-$st and the $2-$nd lifts of a pseudo-Riemannian metric on
a manifold $M$ to the second-order tangent bunde $T^{2}M$ is respectively
given by%
\begin{equation*}
^{0}g=\left( 
\begin{array}{ccc}
g_{ij} & 0 & 0 \\ 
0 & 0 & 0 \\ 
0 & 0 & 0%
\end{array}%
\right)
\end{equation*}%
\begin{equation}
^{I}g=\left( 
\begin{array}{ccc}
y^{s}\partial _{s}g_{ij} & g_{ij} & 0 \\ 
g_{ij} & 0 & 0 \\ 
0 & 0 & 0%
\end{array}%
\right)  \label{AA3.0}
\end{equation}

$\geq $%
\begin{equation}
^{II}g=\left( 
\begin{array}{ccc}
z^{s}\partial _{s}g_{ij}+\frac{1}{2}y^{t}y^{s}\partial _{t}\partial
_{s}g_{ij} & y^{s}\partial _{s}g_{ij} & g_{ij} \\ 
y^{s}\partial _{s}g_{ij} & g_{ij} & 0 \\ 
g_{ij} & 0 & 0%
\end{array}%
\right)  \label{AA3.1}
\end{equation}%
with respect to the induced coordinates $\{\xi ^{A}\}$, where $g_{ij}$
denote local componnents of $g$ on $M$ \cite{YanoIshihara}. By using the $0-$%
th lift of a symmetric $(0,2)$-tensor field $c$ on $(M,g)$ to $T^{2}M$, the
deformed $2-$nd lift metric on $T^{2}M$ is defined by $\overline{g}%
=^{II}g+^{0}c$, that is,%
\begin{equation*}
\overline{g}=^{II}g+^{0}c=\left( 
\begin{array}{ccc}
z^{s}\partial _{s}g_{ij}+\frac{1}{2}y^{t}y^{s}\partial _{t}\partial
_{s}g_{ij}+c_{ij} & y^{s}\partial _{s}g_{ij} & g_{ij} \\ 
y^{s}\partial _{s}g_{ij} & g_{ij} & 0 \\ 
g_{ij} & 0 & 0%
\end{array}%
\right)
\end{equation*}%
with respect to the induced coordinates $\{\xi ^{A}\}$, where $c_{ij}$ are
local components of $c$ on $M$. Also note that the deformed $2-$nd lift
metric is a pseudo-Riemannian metric.

Using (\ref{AA2.1})-(\ref{AA2.3}) and (\ref{AA3.1}), we get:

\begin{proposition}
Let $(M,g)$ be a pseudo-Riemannian manifold and $T^{2}M$ be its second-order
tangent bundle equipped with the deformed $2-$nd lift metric $\overline{g}$.
For any vector field $X,Y$ on $M$,

i) $\overline{g}(^{0}X,^{0}Y)=0,$ ii) $\overline{g}(^{0}X,^{I}Y)=0,$ iii) $%
\overline{g}(^{0}X,^{II}Y)=^{0}(g(X,Y))$

iv) $\overline{g}(^{I}X,^{I}Y)=^{0}(g(X,Y)),$ v) $\overline{g}%
(^{I}X,^{II}Y)=^{I}(g(X,Y)),$

vi) $\overline{g}(^{II}X,^{II}Y)=^{II}(g(X,Y))+^{0}(c(X,Y)).$
\end{proposition}

Let us denote by $L_{\overline{X}}$ the operator of Lie derivation with
respect to any vector field $\overline{X}$ on $T^{2}M$. Consider $L_{%
\overline{X}}\overline{g}$ for arbitrary vector fields $\overline{Y},%
\overline{Z}$ on $T^{2}M$, we have

\begin{eqnarray*}
(L_{^{0}X}\overline{g})(^{II}Y,^{II}Z) &=&L_{^{0}X}(\overline{g}%
(^{II}Y,^{II}Z))-\overline{g}(L_{^{0}X}\text{ }^{II}Y,^{II}Z)-\overline{g}%
(^{II}Y,L_{^{0}X}\text{ }^{II}Z) \\
&=&L_{^{0}X}(^{II}(g(X,Y))+^{0}(c(X,Y)))-\overline{g}(^{0}(L_{X}\text{ }%
Y),^{II}Z) \\
&&-\overline{g}(^{II}Y,^{0}(L_{X}Z)) \\
&=&^{0}(L_{X}(g(Y,Z)))-^{0}(g(L_{X}\text{ }Y,Z))-^{0}(g(Y,L_{X}Z)) \\
&=&^{0}(L_{X}g)(^{II}Y,^{II}Z)
\end{eqnarray*}%
\begin{eqnarray*}
(L_{^{I}X}\overline{g})(^{II}Y,^{II}Z) &=&L_{^{I}X}(\overline{g}%
(^{II}Y,^{II}Z))-\overline{g}(L_{^{I}X}\text{ }^{II}Y,^{II}Z)-\overline{g}%
(^{II}Y,L_{^{I}X}\text{ }^{II}Z) \\
&=&L_{^{I}X}(^{II}(g(Y,Z))+^{0}(c(Y,Z)))-\overline{g}(^{I}(L_{X}Y),^{II}Z) \\
&&-\overline{g}(^{II}Y,^{I}(L_{X}Z)) \\
&=&^{I}(L_{X}(g(X,Y)))-^{I}(g(L_{X}Y,Z))-^{I}(g(Y,L_{X}Z)) \\
&=&^{I}(L_{X}g)(^{II}Y,^{II}Z)
\end{eqnarray*}%
\begin{eqnarray*}
(L_{^{II}X}\overline{g})(^{II}Y,^{II}Z) &=&L_{^{II}X}(\overline{g}%
(^{II}Y,^{II}Z))-\overline{g}(L_{^{II}X}\text{ }^{II}Y,^{II}Z)-\overline{g}%
(^{II}Y,L_{^{II}X}\text{ }^{II}Z) \\
&=&L_{^{II}X}(^{II}(g(Y,Z))+^{0}(c(Y,Z)))-\overline{g}(^{II}(L_{X}Y),^{II}Z)
\\
&&-\overline{g}(^{II}Y,^{II}(L_{X}Z)) \\
&=&^{II}(L_{X}(g(Y,Z)))-^{II}(g(L_{X}Y,Z))-^{II}(g(Y,L_{X}Z)) \\
&=&^{II}(L_{X}g)(^{II}Y,^{II}Z)+^{0}(L_{X}c)(^{II}Y,^{II}Z)
\end{eqnarray*}%
As is known, \ any vector field $X$ on a (pseudo-)Riemannian manifold $(M,g)$
is a Killing vector field if and only if $L_{X}g=0$. Hence, from the
relations above we obtain the following result.

\begin{proposition}
Let $(M,g)$ be a pseudo-Riemannian manifold and $T^{2}M$ be its second-order
tangent bundle equipped with the deformed $2-$nd lift metric $\overline{g}$.

i) The $0-$th and $1-$st lifts of a vector field $X$ on $M$ are both Killing
vector fields on $(T^{2}M,\overline{g})$ if and only if $X$ is a Killing
vector field on $(M,g)$.

ii) The $2-$nd lift of a vector field $X$ on $M$ is a Killing vector field
on $(T^{2}M,\overline{g})$ if and only if $X$ is a Killing vector field on $%
(M,g)$ and $L_{X}c=0$.
\end{proposition}

\section{The Levi-Civita Connection and its curvature tensor field}

Let $M$ be a pseudo-Riemannian manifold with a pseudo-Riemannian metric and $%
\nabla $ be the Levi-Civita connection determined by $g$. Now consider a
global linear connection $\overline{\nabla }$ on $T^{2}M$ denoted by 
\begin{equation}
{\overline{\nabla }}_{\overline{X}}\overline{Y}={}^{II}{{\nabla }_{\overline{%
X}}\overline{Y}}+{}^{0}H(\overline{X},\overline{Y})  \label{AA4.0}
\end{equation}%
for all vector fields $\overline{X},\overline{Y}$ on $T^{2}M$, where $H$ is
a $(1,2)-$tensor field on $M$. The $2-$nd lift of the Levi-Civita connection 
$\nabla $ to $T^{2}M$ satisfies%
\begin{eqnarray*}
^{II}{\nabla }_{{}^{0}X}{}^{0}Y &=&0,^{II}{\nabla }_{{}^{0}X}{}^{I}Y=0,^{II}{%
\nabla }_{{}^{0}X}{}^{II}Y={}^{0}({{\nabla }_{X}}Y), \\
^{II}{\nabla }_{{}^{I}X}{}^{0}Y &=&0,^{II}{\nabla }_{{}^{I}X}{}^{I}Y={}^{0}({%
{\nabla }_{X}}Y),^{II}{\nabla }_{{}^{I}X}{}^{II}Y={}^{I}({{\nabla }_{X}}Y),
\\
^{II}{\nabla }_{{}^{II}X}{}^{0}Y &=&{}^{0}({{\nabla }_{X}}Y),^{II}{\nabla }%
_{{}^{II}X}{}^{I}Y={}^{I}({{\nabla }_{X}}Y),^{II}{\nabla }%
_{{}^{II}X}{}^{II}Y={}^{II}({{\nabla }_{X}}Y)
\end{eqnarray*}%
for all vectors $X,Y$ on $M$ \cite{YanoIshihara}. In view of the relations
above and (\ref{AA4.0}), we now have 
\begin{eqnarray}
\overline{{\nabla }}_{{}^{0}X}{}^{0}Y &=&0,\overline{{\nabla }}%
_{{}^{0}X}{}^{I}Y=0,\overline{{\nabla }}_{{}^{0}X}{}^{II}Y={}^{0}({{\nabla }%
_{X}}Y),  \label{AA4.1} \\
\overline{{\nabla }}_{{}^{I}X}{}^{0}Y &=&0,\overline{{\nabla }}%
_{{}^{I}X}{}^{I}Y={}^{0}({{\nabla }_{X}}Y),\overline{{\nabla }}%
_{{}^{I}X}{}^{II}Y={}^{I}({{\nabla }_{X}}Y),  \notag \\
\overline{{\nabla }}_{{}^{II}X}{}^{0}Y &=&{}^{0}({{\nabla }_{X}}Y),\overline{%
{\nabla }}_{{}^{II}X}{}^{I}Y={}^{I}({{\nabla }_{X}}Y),  \notag \\
\overline{{\nabla }}_{{}^{II}X}{}^{II}Y &=&{}^{II}({{\nabla }_{X}}Y)+{}^{0}(H%
{\left( X,Y\right) }).  \notag
\end{eqnarray}%
Here, we also use the following relations: \cite{YanoIshihara} 
\begin{eqnarray*}
^{0}H(^{0}X,^{0}Y) &=&^{0}H(^{0}X,^{I}Y)=^{0}H(^{0}X,^{II}Y)=0, \\
^{0}H(^{I}X,^{0}Y) &=&^{0}H(^{I}X,^{I}Y)=^{0}H(^{I}X,^{II}Y)=0, \\
^{0}H(^{II}X,^{0}Y) &=&^{0}H(^{II}X,^{I}Y)=0,^{0}H(^{II}X,^{II}Y)=^{0}(H{%
\left( X,Y\right) }).
\end{eqnarray*}

We shall calculate the torsion tensor of the linear connection $\overline{{%
\nabla }}$. The torsion tensor $\overline{T}$ of $\overline{\nabla }$ is, by
definition, given by%
\begin{eqnarray}
\overline{T}\left( {}^{II}X,{}^{II}Y\right) &=&{\overline{\nabla }}%
_{{}^{II}X}{}^{II}Y-{\overline{\nabla }}%
_{{}^{II}Y}{}^{II}X-[{}^{II}X,{}^{II}Y]  \label{AA4.2} \\
&=&{}^{II}({{\nabla }_{X}}Y)+{}^{0}(H{\left( X,Y\right) })\mathrm{-}{}^{II}{%
\left( {\nabla }_{Y}X\right) }  \notag \\
&&-{}^{0}{\left( H\left( Y,X\right) \right) }-{{}^{II}}[X,Y]  \notag \\
&=&{{}^{0}{[H\left( X,Y\right) -H\left( Y,X\right) ].}}  \notag
\end{eqnarray}%
Next, taking covariant derivation of the deformed $2-$nd lift metric $%
\overline{g}$ with respect to the linear connection $\overline{\nabla }$, we
get%
\begin{eqnarray}
&&\left( \overline{{\nabla }}_{{}^{II}X}\overline{g}\right) \left( {}^{II}{%
Y,{}^{II}Z}\right)  \label{AA4.3} \\
&=&{}^{II}X\left( \overline{g}\left( {}^{II}{Y,{}^{II}Z}\right) \right) -%
\overline{g}\left( \overline{{\nabla }}_{{}^{II}X}{}^{II}Y,{}^{II}Z\right) -%
\overline{g}\left( {}^{II}{Y,}\overline{{\nabla }}{_{{}^{II}X}{}^{II}Z}%
\right)  \notag \\
&=&{}^{II}{X\left[ {}^{II}{\left( g\left( Y,Z\right) \right) }+{}^{0}{\left(
c\left( Y,Z\right) \right) }\right] }-\overline{g}\left( {}^{II}{\left( {%
\nabla }_{X}Y\right) }+{}^{0}{\left( H\left( X,Y\right) \right) }%
,{}^{II}Z\right)  \notag \\
&&-\overline{g}\left( {}^{II}Y,{}^{II}{\left( {\nabla }_{X}Z\right) +{}^{0}{%
\left( H\left( X,Z\right) \right) }}\right)  \notag \\
&=&{}^{II}{\left( X\left( g\left( Y,Z\right) \right) \right) }+{}^{0}{\left(
X\left( c\left( Y,Z\right) \right) \right) }-{}^{II}{\left( g\left( {\nabla }%
_{X}Y,Z\right) \right) }-{}^{0}{\left( c\left( {\nabla }_{X}Y,Z\right)
\right) }  \notag \\
&&-{}^{0}{\left( g\left( H\left( X,Y\right) ,Z\right) \right) }-{}^{II}{%
\left( g\left( Y,{\nabla }_{X}Z\right) \right) -{}^{0}{\left( c\left( Y,{%
\nabla }_{X}Z\right) \right) }}-{}^{0}{\left( g\left( Y,H\left( X,Z\right)
\right) \right) }  \notag \\
&=&{}^{0}{\left( -g\left( H\left( X,Y\right) ,Z\right) -g\left( Y,H\left(
X,Z\right) \right) +\left( {\nabla }_{X}c\left( Y,Z\right) \right) \right) .}
\notag
\end{eqnarray}%
If $\overline{{\nabla }}\overline{g}=0$ and $\overline{{\nabla }}$ is
torsion-free, then $\overline{{\nabla }}$ is the Levi-Civita connection of
the deformed $2-$nd lift metric $\overline{g}$. From (\ref{AA4.2}) and (\ref%
{AA4.3}), we find 
\begin{equation}
g\left( H\left( X,Y\right) ,Z\right) =\frac{1}{2}\left[ \left( {\nabla }%
_{X}c\right) \left( Y,Z\right) +\left( {\nabla }_{Y}c\right) \left(
X,Z\right) +\left( {\nabla }_{Z}c\right) \left( X,Y\right) \right] .
\label{AA4.4}
\end{equation}%
Hence, we get

\begin{proposition}
Let $(M,g)$ be a pseudo-Riemannian manifold and $T^{2}M$ be its second-order
tangent bundle equipped with the deformed $2-$nd lift metric $\overline{g}$.
Under the condition (\ref{AA4.4}), the linear connection $\overline{{\nabla }%
}$ is the Levi-Civita connection of $\overline{g}$.
\end{proposition}

Let $X$ be a vector field on $M$ with a linear connection $\nabla $. It is
well-known that $X$ is an affine Killing vector field if and only if $%
L_{X}\nabla =0$. Taking the Lie derivation of the Levi-Civita connection $%
\overline{\nabla }$ with respect to the vector fields $^{0}X$, $^{I}X$ and $%
^{II}X$, we have%
\begin{eqnarray*}
\left( L_{{}^{0}X}\overline{\nabla }\right) \left( {}^{II}Y,{}^{II}Z\right)
&=&L_{{}^{0}X}\left( {\overline{\nabla }}_{{}^{II}Y}{}^{II}Z\right) -{%
\overline{\nabla }}_{{}^{II}Y}\left( L_{{}^{0}X}{}^{II}Z\right) -{\nabla }_{%
\left[ {}^{0}X,{}^{II}Y\right] }{}^{II}Z \\
&=&L_{{}^{0}X}\left( {}^{II}{\left( {\nabla }_{Y}Z\right) +}{}^{0}{\left(
H\left( Y,Z\right) \right) }\right) -{\nabla }_{{}^{II}Y}{}^{0}{\left[ X,Z%
\right] }-{\nabla }_{{}^{0}{\left[ X,Y\right] }}{}^{II}Z \\
&=&L_{{}^{0}X}{}^{II}{\left( {\nabla }_{Y}Z\right) }+L_{{}^{0}X}{}^{0}{%
\left( H\left( Y,Z\right) \right) -{}^{0}{\left( {\nabla }_{Y}\left[ X,Z%
\right] \right) }}-{}^{0}{\left( {\nabla }_{\left[ X,Y\right] }Z\right) } \\
&=&{}^{0}{\left( L_{X}\left( {\nabla }_{Y}Z\right) \right) -}{}^{0}{\left( {%
\nabla }_{Y}\left[ X,Z\right] \right) -{}^{0}{\left( {\nabla }_{\left[ X,Y%
\right] }Z\right) }} \\
&=&{}^{0}{\left( \left( L_{X}\nabla \right) \left( Y,Z\right) \right) }
\end{eqnarray*}%
\begin{eqnarray*}
\left( L_{{}^{I}X}\overline{\nabla }\right) \left( {}^{II}Y,{}^{II}Z\right)
&=&L_{{}^{I}X}\left( {\overline{\nabla }}_{{}^{II}Y}{}^{II}Z\right) -{%
\overline{\nabla }}_{{}^{II}Y}\left( L_{{}^{I}X}{}^{II}Z\right) -{\nabla }_{%
\left[ {}^{I}X,{}^{II}Y\right] }{}^{II}Z \\
&=&L_{{}^{I}X}\left( {}^{II}{\left( {\nabla }_{Y}Z\right) +}{}^{0}{\left(
H\left( Y,Z\right) \right) }\right) -{\nabla }_{{}^{II}Y}{}^{I}{\left[ X,Z%
\right] }-{\nabla }_{{}^{I}{\left[ X,Y\right] }}{}^{II}Z \\
&=&L_{{}^{I}X}{}^{II}{\left( {\nabla }_{Y}Z\right) }+L_{{}^{I}X}{}^{0}{%
\left( H\left( Y,Z\right) \right) -{}^{I}{\left( {\nabla }_{Y}\left[ X,Z%
\right] \right) }}-{}^{I}{\left( {\nabla }_{\left[ X,Y\right] }Z\right) } \\
&=&{}^{I}{\left( L_{X}\left( {\nabla }_{Y}Z\right) \right) -{}^{I}{\left( {%
\nabla }_{Y}\left[ X,Z\right] \right) -{}^{I}{\left( {\nabla }_{\left[ X,Y%
\right] }Z\right) }}} \\
&=&{}^{I}{\left( \left( L_{X}\nabla \right) \left( Y,Z\right) \right) }
\end{eqnarray*}%
\begin{eqnarray*}
\left( L_{{}^{II}X}\overline{\nabla }\right) \left( {}^{II}Y,{}^{II}Z\right)
&=&L_{{}^{II}X}\left( {\overline{\nabla }}_{{}^{II}Y}{}^{II}Z\right) -{%
\overline{\nabla }}_{{}^{II}Y}\left( L_{{}^{II}X}{}^{II}Z\right) \mathrm{-}{%
\nabla }_{\left[ {}^{II}X,{}^{II}Y\right] }{}^{II}Z \\
&=&L_{{}^{II}X}\left( {}^{II}{\left( {\nabla }_{Y}Z\right) +}{}^{0}{\left(
H\left( Y,Z\right) \right) }\right) -{\nabla }_{{}^{II}Y}{}^{II}{\left(
L_{X}Z\right) }-{\nabla }_{{}^{II}{\left[ X,Y\right] }}{}^{II}Z \\
&=&L_{{}^{II}X}{}^{II}{\left( {\nabla }_{Y}Z\right) }+L_{{}^{II}X}{}^{0}{%
\left( H\left( Y,Z\right) \right) -{}^{II}{\left( {\nabla }_{Y}\left(
L_{X}Z\right) \right) }} \\
&&-{}^{0}{\left( H\left( Y,L_{X}Z\right) \right) -{}^{II}{\left( {\nabla }_{%
\left[ X,Y\right] }Z\right) }}-{}^{0}{\left( H\left( \left[ X,Y\right]
,Z\right) \right) } \\
&=&{}^{II}{\left( L_{X}\left( {\nabla }_{Y}Z\right) \right) }-{}^{II}{\left( 
{\nabla }_{Y}\left( L_{X}Z\right) \right) }-{}^{II}{\left( {\nabla }_{\left[
X,Y\right] }Z\right) }+L_{X}{}^{0}{\left( H\left( Y,Z\right) \right) } \\
&&{-{}^{0}{\left( H\left( Y,L_{X}Z\right) \right) -{}^{0}{\left( H\left(
L_{X}Y,Z\right) \right) }}} \\
&=&{}^{II}{\left( \left( L_{X}\nabla \right) \left( Y,Z\right) \right)
+{}^{0}{\left\{ L_{X}\left( H\left( Y,Z\right) \right) -H\left(
Y,L_{X}Z\right) -H\left( L_{X}Y,Z\right) \right\} }} \\
&=&{}^{II}{\left( \left( L_{X}\nabla \right) \left( Y,Z\right) \right) }%
+{}^{0}{\left( \left( L_{X}H\right) \left( Y,Z\right) \right) .}
\end{eqnarray*}%
Thus, the above relations give the following.

\begin{proposition}
Let $(M,g)$ be a pseudo-Riemannian manifold and $T^{2}M$ be its second-order
tangent bundle equipped with the deformed $2-$nd lift metric $\overline{g}$.

i) The $0-$th and $1-$st lifts of a vector field $X$ on $M$ are both affine
Killing vector fields on $(T^{2}M,\overline{g})$ if and only if $X$ is an
affine Killing vector field on $(M,g)$.

ii) The $2-$nd lift of a vector field $X$ on $M$ is an affine Killing vector
field on $(T^{2}M,\overline{g})$ if and only if $X$ is an affine Killing
vector field on $(M,g)$ and $L_{X}H=0$.
\end{proposition}

For the Riemannian curvature tensor field $\overline{R}$ of the Levi-Civita
connection $\overline{\nabla }$, we obtain%
\begin{eqnarray*}
&&\overline{R}\left( {}^{II}X,{}^{II}Y\right) {{}^{II}{Z}} \\
{} &{{=\overline{\nabla }}}&_{{}^{II}X}{\overline{\nabla }}%
_{{}^{II}Y}{}^{II}Z-{\overline{\nabla }}_{{}^{II}Y}{\overline{\nabla }}%
_{{}^{II}X}{}^{II}Z\mathrm{-}{\overline{\nabla }}_{\left[ {}^{II}X,{}^{II}Y%
\right] }{}^{II}Z \\
&=&{\overline{\nabla }}_{{}^{II}X}\left[ {}^{II}{\left( {\nabla }%
_{Y}Z\right) }+{{}^{0}{\left( H\left( Y,Z\right) \right) }}\right] \\
&&-{\overline{\nabla }}_{{}^{II}Y}\left[ {}^{II}{\left( {\nabla }%
_{X}Z\right) }+{{}^{0}{\left( H\left( X,Z\right) \right) }}\right] -{%
\overline{\nabla }}_{{}^{II}{\left[ X,Y\right] }}{}^{II}Z \\
&{=}&{\overline{\nabla }}_{{}^{II}X}{}^{II}{\left( {\nabla }_{Y}Z\right) }+{%
\overline{\nabla }}_{{}^{II}X}{}^{0}{\left( H\left( Y,Z\right) \right) }-{%
\overline{\nabla }}_{{}^{II}Y}{}^{II}{\left( {\nabla }_{X}Z\right) } \\
&&-{{\overline{\nabla }}_{{}^{II}Y}{}^{0}{\left( H\left( X,Z\right) \right) }%
}-{}^{II}{\left( {\nabla }_{\left[ X,Y\right] }Z\right) }-{}^{0}{\left(
H\left( \left[ X,Y\right] ,Z\right) \right) } \\
&=&{}^{II}{\left( {\nabla }_{X}{\nabla }_{Y}Z\right) +{}^{0}{\left( H\left(
X,{\nabla }_{Y}Z\right) \right) }}+{}^{0}{\left( {\nabla }_{X}H\left(
Y,Z\right) \right) -{}^{II}{\left( {\nabla }_{Y}{\nabla }_{X}Z\right) }} \\
&&-{}^{0}{\left( H\left( Y,{\nabla }_{X}Z\right) \right) }-{}^{0}{\left( {%
\nabla }_{Y}H(X,Z)\right) }-{}^{II}{\left( {\nabla }_{\left[ X,Y\right]
}Z\right) -{}^{0}{\left( H\left( \left[ X,Y\right] ,Z\right) \right) }} \\
&=&{}^{II}{\left( R\left( X,Y\right) Z\right) +{}^{0}{\left\{ 
\begin{array}{c}
H\left( X,{\nabla }_{Y}Z\right) +{\nabla }_{X}H\left( Y,Z\right) \\ 
-H\left( Y,{\nabla }_{X}Z\right) -{\nabla }_{Y}H\left( X,Z\right) -H\left( %
\left[ X,Y\right] ,Z\right)%
\end{array}%
\right\} }}
\end{eqnarray*}%
from which, using ${\nabla }_{X}Y-{\nabla }_{Y}X=\left[ X,Y\right] $ we get 
\begin{equation*}
\overline{R}\left( {}^{II}{X,{}^{II}Y}\right) {}^{II}Z={}^{II}{\left(
R\left( X,Y\right) Z\right) }+{}^{0}{\left\{ \left( {\nabla }_{X}H\right)
\left( Y,Z\right) -\left( {\nabla }_{Y}H\right) \left( X,Z\right) \right\} .}
\end{equation*}%
Thus we state following result.

\begin{proposition}
Let $(M,g)$ be a pseudo-Riemannian manifold and $T^{2}M$ be its second-order
tangent bundle equipped with the deformed $2-$nd lift metric $\overline{g}$. 
$(T^{2}M,\overline{g})$ is flat if and if the base manifold $M$ is flat and
the condition $\left( {\nabla }_{X}H\right) \left( Y,Z\right) =\left( {%
\nabla }_{Y}H\right) \left( X,Z\right) $ is fulfilled.
\end{proposition}

The $(0,4)-$Riemannian curvature tensor of the Levi-Civita connection $%
\overline{\nabla }$ is as follows: 
\begin{eqnarray}
&&\tilde{R}\left( {}^{II}{X,{}^{II}Y,{}^{II}{Z,}}{}^{II}W\right)
\label{AA4.5} \\
&=&\overline{g}\left( \overline{R}\left( {}^{II}X,{}^{II}Y\right) {}^{II}{Z,}%
{}^{II}W\right)  \notag \\
&=&\overline{g}\left( {}^{II}{\left( R\left( X,Y\right) Z\right) }%
,{}^{II}W\right) +\overline{g}\left( {}^{0}{\left[ \left( {\nabla }%
_{X}H\right) \left( Y,Z\right) -\left( {\nabla }_{Y}H\right) \left(
X,Z\right) \right] },{}^{II}W\right)  \notag \\
&=&{}^{II}{\left( g\left( R\left( X,Y\right) Z,W\right) \right) +{}^{0}{%
\left( c\left( R\left( X,Y\right) Z,W\right) \right) }}  \notag \\
&&{{+{}^{0}{\left( g\left( \left( {\nabla }_{X}H\right) \left( Y,Z\right)
-\left( {\nabla }_{Y}H\right) \left( X,Z\right) ,W\right) \right) }}}  \notag
\\
&=&{}^{II}{\left( R\left( X,Y,Z,W\right) \right) }+{}^{0}{\left( c\left(
R\left( X,Y\right) Z,W\right) \right) }  \notag \\
&&{+}{}^{0}{\left[ \left( {\nabla }_{X}H\right) \left( Y,Z,W\right) -\left( {%
\nabla }_{Y}H\right) \left( X,Z,W\right) \right] .}  \notag
\end{eqnarray}

Given a manifold $M$ $(\dim (M)\geq 3)$ endowed with a linear connection $%
\nabla $ whose curvature tensor is signed as $R$, for any tensor field of $S$
of type $(0,k),k\geq 1,$ the tensor field $R(X,Y).S$ is expressed in the
form:%
\begin{eqnarray*}
(R(X,Y).S)(X_{1},X_{2},...,X_{k}) &=&-S(R(X,Y)X_{1},X_{2},...,X_{k}) \\
&&-...-S(X_{1},X_{2},...,X_{k-1},R(X,Y)X_{k})
\end{eqnarray*}%
for any vector fields $X_{1},X_{2},...,X_{k},X,Y$ on $M$, where $R(X,Y)$
acts as a derivation on $S$. If $R(X,Y).S=0$, then the manifold $M$ is said
to be $S$ semi-symmetric with respect to the linear connection $\nabla $. A
(pseudo-) Riemannian manifold $(M,g)$ such that its curvature tensor $R$
satisfies the condition 
\begin{equation*}
R(X,Y).R=0
\end{equation*}%
is called a semi-symmetric manifold. Also, note that locally symmetric
manifold $(\nabla R=0)$ are semi-symmetric, but in general the converse is
not true. The semi-symmetric manifold was first studied by Cartan.
Nevertheless, Sinjukov first used the name "semi-symmetric" for manifolds
satisfying the above curvature condition \cite{Sinjukov}. Later, Szabo gave
the full local and global classification of semi-symmetric manifolds \cite%
{Szabo1,Szabo2}. Now we are interested in the semi-symmetry property of $%
T^{2}M$ with the deformed $2-$nd lift metric $\overline{g}$. \noindent For
the sake of simplicity we shall choose $c=g$ in the deformed $2-$nd lift
metric $\overline{g}$. In this case, the relation (\ref{AA4.5}) reduces to%
\begin{equation*}
\tilde{R}\left( {}^{II}{X,{}^{II}Y,{}^{II}{Z,}}{}^{II}W\right) ={}^{II}{%
\left( R\left( X,Y,Z,W\right) \right) }+{}^{0}{\left( R\left( X,Y,Z,W\right)
\right) .}
\end{equation*}

\begin{theorem}
Let $(M,g)$ be a pseudo-Riemannian manifold and $T^{2}M$ be its second-order
tangent bundle equipped with the deformed $2-$nd lift metric $\overline{g}%
=^{II}g+^{0}g$. $(T^{2}M,\overline{g})$ is semi-symmetric if and only if $%
(M,g)$ is semi-symmetric.
\end{theorem}

\begin{proof}
We consider the condition $\overline{R}(\overline{X},\overline{Y}).%
\widetilde{R}=0$ for all vector field $\overline{X}$, $\overline{Y}$ on $%
(T^{2}M,\overline{g})$. \ We calculate%
\begin{eqnarray*}
&&\left( \overline{R}({{^{II}X}},{}^{II}Y).\tilde{R}\right) \left( {}^{II}{%
X_{1},}{}^{II}{X_{2},}{}^{II}{X_{3},{}^{II}{X_{4}}}\right) \\
&=&-\tilde{R}\left( \overline{R}\left( {}^{II}{X,{}^{II}Y}\right) {}^{II}{%
X_{1}},{}^{II}{X_{2}},{}^{II}{X_{3,}{}^{II}{X_{4}}}\right) -\tilde{R}\left(
{}^{II}{X_{1}},\overline{R}\left( {}^{II}{X,{}^{II}Y}\right) {}^{II}{X_{2}}%
,{}^{II}{X_{3,}{}^{II}{X_{4}}}\right) \\
&&-\tilde{R}\left( {}^{II}{X_{1}},{}^{II}{X_{2}},\overline{R}\left( {}^{II}{%
X,{}^{II}Y}\right) {}^{II}{X_{3,}{}^{II}{X_{4}}}\right) -\tilde{R}\left(
{}^{II}{X_{1}},{}^{II}{X_{2}},{}^{II}{X_{3,}\overline{R}\left( {}^{II}{%
X,{}^{II}Y}\right) {}^{II}{X_{4}}}\right)
\end{eqnarray*}%
\begin{eqnarray*}
&=&-\tilde{R}\left( {}^{II}{\left( \overline{R}\left( X,Y\right)
X_{1}\right) },{}^{II}{X_{2}},{}^{II}{X_{3,}{}^{II}{X_{4}}}\right) -\tilde{R}%
\left( {}^{II}{X_{1}},{}^{II}{\left( \overline{R}\left( X,Y\right)
X_{2}\right) },{}^{II}{X_{3,}{}^{II}{X_{4}}}\right) \\
&&-\tilde{R}\left( {}^{II}{X_{1}},{}^{II}{X_{2}},{}^{II}{\left( \overline{R}%
\left( X,Y\right) X_{3}\right) ,{}^{II}{X_{4}}}\right) -\tilde{R}\left(
{}^{II}{X_{1}},{}^{II}{X_{2}},{}^{II}{X_{3,}},{}^{II}{\left( \overline{R}%
\left( X,Y\right) X_{4}\right) }\right)
\end{eqnarray*}%
\begin{eqnarray*}
&=&-{}^{II}{\left[ R\left( R\left( X,Y\right) X_{1},X_{2},X_{3},X_{4}\right) %
\right] -{}^{0}{\left[ R\left( R\left( X,Y\right)
X_{1},X_{2},X_{3},X_{4}\right) \right] }} \\
&&-{}^{II}{\left[ R\left( X_{1},R\left( X,Y\right) X_{2},X_{3},X_{4}\right) %
\right] -{}^{0}{\left[ R\left( X_{1},R\left( X,Y\right)
X_{2},X_{3},X_{4}\right) \right] }} \\
&&-{}^{II}{\left[ R\left( X_{1},X_{2},{R\left( X,Y\right) X}%
_{3},X_{4}\right) \right] -{}^{0}{\left[ R\left( X_{1},X_{2},R\left(
X,Y\right) X_{3},X_{4}\right) \right] }} \\
&&-{}^{II}{\left[ R\left( X_{1},X_{2},X_{3},{R\left( X,Y\right) X}%
_{4}\right) \right] -{}^{0}{\left[ R\left( X_{1},X_{2},X_{3},R\left(
X,Y\right) X_{4}\right) \right] }}
\end{eqnarray*}%
\begin{equation*}
={}^{II}{\left\{ \left( R(X,Y).R\right) \left(
X_{1},X_{2},X_{3},X_{4}\right) \right\} }+{}^{0}{\left\{ \left(
R(X,Y).R\right) \left( X_{1},X_{2},X_{3},X_{4}\right) \right\} }
\end{equation*}%
which completes proof.
\end{proof}

\section{\protect\bigskip Plural-holomorphic $B-$manifolds}

A nilpotent structure on $M$ is a $(1,1)-$tensor field $\gamma $ such that $%
\gamma ^{3}=0$ $(\gamma \neq 0)$. A pure metric (for pure tensors, see \cite%
{Salimov3}) with respect to the nilpotent structure is a pseudo-Riemannian
metric $g$ such that 
\begin{equation*}
g(\gamma X,Y)=g(X,\gamma Y)
\end{equation*}%
for any vector fields $X,Y$ on $M$. Metrics of this type have also been
studied under the name $B$-metrics \cite{V1,V2,V3,V5}, since the metric
tensor $g$ with respect to the structure $\gamma $ is $B$-tensor according
to the terminology accepted in \cite{N}. If $\left( M,\gamma \right) $ is a
manifold with a $B$-metric and a nilpotent structure, we say that $\left(
M,\gamma ,g\right) $ is an almost $B$-manifold. If $\gamma $ is integrable,
we say that $(M,\gamma ,g)$ is a $B$-manifold. A plural-holomorphic $B-$%
manifold \cite{IA} can be defined as a triple $(M,\gamma ,g)$ which consists
of a smooth manifold $M$ endowed with a nilpotent structure $\gamma $ and a $%
B$-metric $g$ such that $\Phi _{\gamma }g=0$, where $\Phi _{\gamma }$ is the
Tachibana operator \cite{Tachibana,YanoAko}: $(\Phi _{\gamma
}g)(X,Y,Z)=(\gamma X)({}g(Y,Z))-X(g(\gamma Y,Z))+g((L_{Y}\gamma
)X,Z)+g(Y,(L_{Z}$ $\gamma )X)$.

Recall that there exists a $(1,1)-$tensor field $\widehat{\gamma }$ on $%
T^{2}M$ which has components of the form%
\begin{equation*}
\widehat{\gamma }=\left( 
\begin{array}{ccc}
0 & 0 & 0 \\ 
I & 0 & 0 \\ 
0 & I & 0%
\end{array}%
\right) ~
\end{equation*}%
with respect to the induced coordinates $\{\xi ^{A}\}$, where $I$ being unit
matrix. The tensor field satisfies $\widehat{\gamma }^{3}=0$, that is,0 $%
T^{2}M$ has a natural integrable nilpotent structure. The natural integrable
nilpotent structure has the properties%
\begin{equation*}
\widehat{\gamma }\text{ }^{0}X=0,\widehat{\gamma }\text{ }^{I}X=^{0}X,%
\widehat{\gamma }\text{ }^{II}X=\text{ }^{I}X
\end{equation*}%
which characterize $\widehat{\gamma }$. We compute, for any vector fields $%
X,Y$ on $M$%
\begin{equation*}
\overline{g}(\widehat{\gamma }\text{ }^{II}X,^{II}Y)=\overline{g}%
(^{I}X,^{II}Y)=^{I}(g(X,Y))
\end{equation*}%
\begin{equation*}
\overline{g}(^{II}X,\widehat{\gamma }^{II}Y)=\overline{g}%
(^{II}X,^{I}Y)=^{I}(g(X,Y))
\end{equation*}%
that is, the the deformed $2-$nd lift metric $\overline{g}$ is a $B-$metric
with respect to $\widehat{\gamma }$. Hence $(T^{2}M,\widehat{\gamma },%
\overline{g})$ is a $B-$manifold. Applying the Tachibana operator $\Phi _{%
\widehat{\gamma }}$ to $\overline{g},$ we get%
\begin{eqnarray*}
(\Phi _{\widehat{\gamma }}\overline{g})(^{II}X,^{II}Y,^{II}Z) &=&(\widehat{%
\gamma }^{II}X)({}\overline{g}(^{II}Y,^{II}Z))-^{II}X(\overline{g}(\widehat{%
\gamma }^{II}Y,^{II}Z)) \\
&&+\overline{g}((L_{^{II}Y}\widehat{\gamma })^{II}X,^{II}Z)+\overline{g}%
(^{II}Y,(L_{^{II}Z}\widehat{\gamma })^{II}X) \\
&=&^{I}X^{II}(g(Y,Z))-^{II}X^{I}(g(Y,Z)) \\
&=&^{I}(X(g(Y,Z)))-^{I}(X(g(Y,Z))) \\
&=&0
\end{eqnarray*}

Hence we state the following theorem.

\begin{theorem}
Let $(M,g)$ be a pseudo-Riemannian manifold and $T^{2}M$ be its second-order
tangent bundle equipped with the deformed $2-$nd lift metric $\overline{g}$
and the natural integrable nilpotent structure $\widehat{\gamma }$. The
triple $(T^{2}M,\widehat{\gamma },\overline{g})$ is a plural-holomorphic $B-$%
manifold.
\end{theorem}

\section{Anti-K\"{a}hler structures on $T^{2}M$}

An almost complex anti-Hermitian manifold $(M,J,g)$ is a real $2k-$%
dimensional differentiable manifold $M$ with an almost complex structure $J$
and a pseudo-Riemannian metric $g$ such that:%
\begin{equation*}
g(JX,Y)=g(X,JY)
\end{equation*}%
for all vector fields $X,Y$ on $M$. An anti-K\"{a}hler manifold can be
defined as a triple $(M,J,g)$ which consists of a smooth manifold $M$
endowed with an almost complex structure $J$ and an anti-Hermitian metric $g$
such that $\nabla J=0$, where $\nabla $ is the Levi-Civita connection of $g$%
. It is well known that the condition $\nabla J=0$ is equivalent to $\mathrm{%
C}$-holomorphicity (analyticity) of the anti-Hermitian metric $g$, i.e. $%
\Phi _{J}g=0$ \cite{Iscan}. Since in dimension $2$ an anti-K\"{a}hler
manifold is flat, we assume in the sequel that $\dim \,M\geq 4$.

The $2-$nd lift of a $(1,1)$-tensor field $J$ to $T^{2}M$ has the followngs%
\begin{equation}
^{II}J(^{II}X)=^{II}(JX)\text{, }^{II}J({}^{I}X)=^{I}(JX)\text{, }%
^{II}J({}^{0}X)=^{0}(JX)  \label{AA4.6}
\end{equation}%
for any $X$ on $M$. Moreover, it is well known that \i f $J$ is an almost
complex structure on $(M,g)$, then $^{II}J$ is an almost complex structure
on $T^{2}M$ \cite{YanoIshihara}. Now we prove the following theorem.

\begin{theorem}
Let $(M,J,g)$ be an anti-K\"{a}hler manifold. Then $T^{2}M$ is an anti-K\"{a}%
hler manifold equipped with the deformed $2-$nd lift metric $\overline{g}$
and the almost complex structure $^{II}J$ if and only if the$\ $symmetric $%
(0,2)-$tensor field $c$ on $M$ is a holomorphic tensor field with respect to
the almost complex structure $J$.
\end{theorem}

\begin{proof}
Let $(M,J,g)$ be a anti-K\"{a}hler manifold. Then we have 
\begin{eqnarray*}
&&{}\overline{g}\left( ^{II}J(^{II}X),^{II}Y\right) -{}\overline{g}\left(
^{II}X,^{II}J(^{II}Y)\right) \\
&=&{}\overline{g}\left( ^{II}(JX),^{II}Y\right) -{}\overline{g}\left(
^{II}X,^{II}(JY)\right) \\
&=&c(JX,Y)-c(X,JY).
\end{eqnarray*}%
From the last equations, the deformed $2-$nd lift metric $\overline{g}$ is
anti-Hermitian with respect to $^{II}J$ if and only if the$\ $symmetric $%
(0,2)-$tensor field $c$ is pure with respect to $J$.

Now, we are interested in the holomorphy property of the deformed $2-$nd
lift metric $\overline{g}$ with respect to $^{II}$ $J$. We calculate 
\begin{eqnarray*}
&&(\Phi _{^{II}J}{}\overline{g})(^{II}X,^{II}Y,^{II}Z) \\
&=&(^{II}J^{II}X)(\overline{g}(^{II}Y,^{II}Z))-^{II}X(\overline{g}%
(^{II}J^{II}Y,^{II}Z)) \\
&+&{}\overline{g}((L_{^{II}Y}\text{ }^{II}J)^{II}X,^{II}Z)+{}\overline{g}%
(^{II}Y,(L_{^{II}Z}\text{ }^{II}J)^{II}X) \\
&=&^{II}(JX)\{^{II}(g(X,Y))+^{0}(c(X,Y))\}-^{II}X%
\{^{II}(g(JY,Z))+^{0}(c(JY,Z))\} \\
&&+^{II}\{g((L_{Y}\text{ }J)X,Z)\}+^{0}\{c((L_{Y}\text{ }J)X,Z)\}+^{II}%
\{g(Y,(L_{Z}\text{ }J)X)\}+^{0}\{c(Y,(L_{Z}\text{ }J)X)\} \\
&=&^{II}\{(JX)(g(X,Y))-X(g(JY,Z))+g((L_{Y}\text{ }J)X,Z)+g(Y,(L_{Z}\text{ }%
J)X)\} \\
&&+^{0}\{JX)(c(X,Y))-X(c(JY,Z))+c((L_{Y}\text{ }J)X,Z)+c(Y,(L_{Z}\text{ }%
J)X)\} \\
&=&^{II}\{(\Phi _{J}{}g)(X,Y,Z)\}+^{II}\{(\Phi _{J}{}c)(X,Y,Z)\}.
\end{eqnarray*}%
Hence, from the relation above, since $(\Phi _{J}{}g)=0$, it follows that $%
\Phi _{^{II}J}{}{}\overline{g}=0$ if and only if $\Phi _{J}{}c=0$, that is, $%
c$ is holomorphic. This completes the proof.
\end{proof}


\begin{thebibliography}{99}
\bibitem{Leon} M. de Leon , E. Vazquez, On the geometry of the tangent
bundle of order 2. An. Univ. Bucure\c{s}ti Mat. \textbf{34} (1985), 40--48.

\bibitem{N} A. P. Norden, On a certain class of four-dimensional A-spaces.
Izv. Vuzov. Mat. \textbf{4} (1960), 145--157.

\bibitem{Iscan} M. Iscan, A. A. Salimov, On K\"{a}hler-Norden manifolds.
Proc. Indian Acad. Sci. Math. Sci. \textbf{119} (2009), no. 1, 71--80.

\bibitem{IA} M. Iscan, A. Magden, On B-manifolds defined by algebra of
plural numbers. Arab. J. Sci. Eng. \textbf{35} (2010), Number 1D, 57-63.

\bibitem{Dodson} C.T.J. Dodson and M.S. Radivoiovici , Tangent and frame
bundles of order two. Analele stiintifice ale Universitatii
\textquotedblright Al. I. Cuza\textquotedblright\ \textbf{28} (1982), 63-71.

\bibitem{Salimov3} A. Salimov, On operators associated with tensor fields,
J. Geom. \textbf{99 }(1--2) (2010), 107--145.

\bibitem{Sinjukov} N. S. Sinjukov,\emph{\ }Geodesic mappings of Riemannian
spaces (Russian). Publishing House \textquotedblleft Nauka\textquotedblright
, Moscow, 1979.

\bibitem{Szabo1} Z. I. Szabo, Structure theorems on Riemannian spaces
satisfying\emph{\ }$R(X,Y).R=0$. I. The local version, J. Differential Geom. 
\textbf{17} (1982), 531--582.

\bibitem{Szabo2} Z. I. Szabo, Structure theorems on Riemannian spaces
satisfying\emph{\ }$R(X,Y).R=0$. II. Global version, Geom. Dedicata \textbf{%
19} (1985), 65-108.

\bibitem{Tachibana} S. Tachibana , Analytic tensor and its generalization%
\emph{.} Tohoku Math. J. \textbf{12} (1960) no.2 208-221.

\bibitem{V1} V. V. Vishnevskii, Structures of projective spaces generated by
affinor, Izv. Vuzov. Mat. \textbf{6 }(1969), 35--46.

\bibitem{V2} V. V. Vishnevskii, Affinor structures of affine connection
spaces, Izv. Vuzov. Mat. \textbf{1} (1970), 12--23.

\bibitem{V3} V. V. Vishnevskii, Integrable affinor structures and their
plural interpretations, J. Math. Sci. \textbf{108} (2) (2002), 151--187.

\bibitem{V5} V. V. Vishnevskii, A. P. Shirokov, V. V. Shurygin, Spaces over
algebras. Kazan Gos. University, Kazan, Russian, 1985.

\bibitem{YanoIshihara} K. Yano , S. Ishihara , Tangent and cotangent
bundles: differential geometry. Pure and Applied Mathematics, No. 16. Marcel
Dekker, Inc., New York, 1973.

\bibitem{YanoAko} K. Yano, M. Ako, On certain operators associated with
tensor field\emph{,} Kodai Math. Sem. Rep. 20 (1968) 414-436.
\end{thebibliography}
\end{document}